\documentclass[12pt]{article}
\usepackage{amssymb,latexsym}
\usepackage{amscd}
\usepackage{euscript}
\usepackage{eufrak}
\usepackage{amsmath}
\usepackage{color}
\usepackage{tikz-cd}
\usepackage[all,cmtip]{xy}

\newtheorem{definition}{\bf Definition}[section]
\newtheorem{remark}{\bf Remark}[section]
\newtheorem{lemma}{\bf Lemma}[section]
\newtheorem{theorem}{\bf Theorem}[section]
\newtheorem{corollary}{\bf Corollary}[section]
\newtheorem{proposition}{\bf Proposition}[section]
\newtheorem{example}{\bf Example}[section]
\newtheorem{cor}{\bf Corollary}[section]

\newenvironment{proof}{\bf Proof. \rm}{\hfill $\mbox{\boldmath{$\square$}}$}

\title{A Stone-type duality for semilattices with adjunctions}

\author{B. Gimenez, G. Pelaitay, W. Zuluaga}

\date{}

\begin{document}

\maketitle

\begin{abstract}
\noindent
This paper focuses on semilattices with adjunctions (SLatas), which are semilattices with a greatest element enriched with a pair of adjoint maps. We develop a spectral-style duality for SLatas, building on prior topological dualities for monotone semilattices. As an application of this duality, we characterize SLata congruences through an adaptation of lower-Vietoris-type families. Furthermore, we investigate tense operators on semilattices with a greatest element, deriving a duality for this case by extending the results obtained for SLatas.
\end{abstract}

\font\fivrm=cmr5 \relax

\thispagestyle{empty}

\vspace*{-6mm}

\vspace{2cm}


\section{Introduction}\label{Introduction}

Galois connections arise in the study of relationships between two structured domains, where specific functions enable transformations between them. When these mappings are applied back and forth, stability is achieved. Often, these transformations respect the natural order relations of the underlying structures, simplifying their definition and yielding a wealth of results in this context. In 1944, O. Ore \cite{ore1944} formalized the concept of a Galois connection as a pair of functions that reverse the order between two partially ordered sets, generalizing Birkhoff's \cite{birkhoff} theory of polarities among complete lattices. This concept extended the correspondence between subgroups and subfields in the famous Fundamental Theorem of Galois Theory, giving rise to the term ``Galois connection.'' Subsequently, in 1949, J. Schmidt \cite{schmidt} introduced the concept of order-preserving Galois connections, also known as adjoint pairs, which have been studied in lattice theory as residual maps. The covariant form of Galois connections is particularly advantageous for simplifying the composition of these mappings and has found significant applications in computer science, where preserving relative information is crucial—ensuring that the underlying order relations remain isomorphic when comparing (partially) equivalent structures.

Later, in 1958, Daniel M. Kan applied category theory to homotopy \cite{K1958}, introducing and abstractly developing the concept of adjoint functors with the geometric realization of simplicial sets in mind. It was later realized that, when partially ordered sets are viewed as categories, covariant Galois connections coincide with adjoint functors.

This interrelationship underscores the utility of Galois connections across mathematics, computer science, and logic, while also highlighting their significance in category theory and structural analysis. Additionally, this framework has been applied to various algebraic structures, such as Hilbert algebras, distributive lattices, and Heyting algebras, among other ordered algebraic frameworks (see \cite{Celani,ChajdaPaseka1,ChajdaPaseka,Dzik10,8,9,10}). In \cite{CG}, Celani and González developed a categorical duality between the category of semilattices with homomorphisms and a category of certain multi-relational topological spaces with meet-relations as morphisms. This duality for semilattices naturally extends Stone's duality for lattices via spectral spaces \cite{stone1937}. Building on these results, Calomino, Menchón, and Zuluaga \cite{CMZ} later established a spectral-style multi-relational topological duality for semilattices with monotone operators.

In this paper, we introduce the class of SLatas, which are semilattices with a greatest element enriched with an adjoint pair. Since an SLata is a semilattice equipped with two monotone operators, we leverage these dualities to develop a spectral-style duality specifically for SLatas. This constitutes the primary goal of this paper and is detailed in Section \ref{Semilattices with an adjunction}.

As a secondary goal, we apply the developed duality to study SLata congruences via the so-called SLata-Vietoris families—an adaptation of the monotone lower-Vietoris-type families used in \cite{CMZ}, which in turn are inspired in the ideas of \cite{ID2017}. This is discussed in Section \ref{Congruences and SLata-Vietoris Families}. Lastly, recognizing the potential applications of SLatas in tense logics, Section \ref{Ewald-Semilattices with adjunctions} introduces Ewald-semilattices with adjunctions to study tense operators on semilattices with a greatest element. For this case, we also provide a spectral-style duality.

\section{Preliminaries}\label{Preliminaries}

To make this paper as self-contained as possible, this section provides the necessary definitions and introduces the notation required for the reader to follow the content.

\ 

Let $f : X \to Y$ be a function. For any subset $U \subseteq X$, the direct image of $U$ under $f$ is denoted by $f[U]$. Similarly, for any subset $V \subseteq Y$, the inverse image of $V$ under $f$ is denoted by $f^{-1}[V]$. If $g : Y \to Z$ is another function, the composition of $g$ with $f$ is written as $g \circ f$. Given a binary relation $R \subseteq X \times Y$ and an element $x \in X$, we define $R(x) = \{y \in Y : (x, y) \in R\}$.  
If $T \subseteq Y \times Z$ is another binary relation, their composition $T \circ R$ is defined by  
\[
T \circ R = \{(x, z) : \exists y \in Y \, [(x, y) \in R \text{ and } (y, z) \in T]\}.
\]

Let $\langle X, \leq \rangle$ be a poset. For any subset $Y \subseteq X$, define:  
\[
[Y) = \{x \in X : \exists y \in Y \, (y \leq x)\}, \quad (Y] = \{x \in X : \exists y \in Y \, (x \leq y)\}.
\]  
A subset $Y$ of $X$ is called an \emph{upset} if $Y = [Y)$, and a \emph{downset} if $Y = (Y]$.  
We denote by $\mathcal{P}(X)$ the power set of $X$, and by $\text{Up}(X)$ the collection of all upsets of $X$.  
The complement of a subset $Y \subseteq X$ is denoted by $Y^c$.
\\

A \textit{meet-semilattice} with a greatest element, or simply \textit{semilattice}, is an algebra $\mathbf{A} = \langle A, \wedge, 1\rangle$ of type $(2, 0)$ such that the operation $\wedge$ is idempotent, commutative, associative, and $a \wedge 1 = a$ for all $a \in A$. The partial order $\leq$ on $\mathbf{A}$ is defined by $a \leq b$ if and only if $a = a \wedge b$. For each partially ordered set $(X, \leq)$, the structure $\langle\text{Up}(X), \cap, X\rangle$ is a \textit{semilattice}. 

If $\mathbf{A}$ is a semilattice, a subset $F \subseteq A$ is a \emph{filter} of $\mathbf{A}$ if it is an upset, $1 \in F$ and if $a, b \in F$, then $a \wedge b \in F$. The set of all filters of $\mathbf{A}$ will be denoted by $\text{Fi}(\mathbf{A})$. The \emph{filter generated} by the subset $X \subseteq A$ will be denoted by $\text{F}(X)$. If $X = \{a\}$, then $\text{F}(\{a\}) = [\{a\})$, or simply, $[a)$. We say that a proper filter $F \in \text{Fi}(\mathbf{A})$ is \emph{irreducible} if for all $F_1, F_2 \in \text{Fi}(\mathbf{A})$, if $F = F_1 \cap F_2$, then $F = F_1$ or $F = F_2$. We write $\mathcal{X}(\mathbf{A})$ for the set of all irreducible filters of $\mathbf{A}$.

Let $\mathbf{A}$ and $\mathbf{B}$ be two semilattices. A map $h : A \to B$ is a homomorphism if $h(1) = 1$ and $h(a \wedge b) = h(a) \wedge h(b)$ for all $a, b \in A$. If we consider the poset $(\mathcal{X}(\mathbf{A}), \subseteq)$ and the map $\beta_{\mathbf{A}} : \mathbf{A} \to \text{Up}(\mathcal{X}(\mathbf{A}))$ given by $\beta_{\mathbf{A}}(a) = \{P \in \mathcal{X}(\mathbf{A}) : a \in P\}$, then it is proved that $\mathbf{A}$ is isomorphic to the subalgebra of $(\text{Up}(\mathcal{X}(\mathbf{A})), \cap, \mathcal{X}(\mathbf{A}))$ whose universe is $\beta_{\mathbf{A}}[A] = \{\beta_\mathbf{A}(a) : a \in A\}$. Throughout the paper and to simplify notation, we will omit the subscript of $\beta_\mathbf{A}$ where appropriate.

If $(X, \tau)$ is a topological space and $Y \subseteq X$, we write $\text{cl}(Y)$ for the topological closure of $Y$. In particular, if $Y = \{y\}$, then we simply denote $\text{cl}(\{y\})$ by $\text{cl}(y)$. It is known that every topological space can be endowed with a partial order defined by $x \leq y$ if and only if $x \in \text{cl}(y)$. Such an order is called the \emph{specialization order} of $X$.
\\

Along this paper, by \textit{topological space} we mean a pair $(X, \mathcal{K})$ where $(X, \tau)$ is a topological space and $\mathcal{K}$ is a subbase for $\tau$ (i.e., $\mathcal{K} \subseteq X$ and $X = \bigcup \mathcal{K}$). We consider the following family of subsets of $X$:
\[
S(X) = \{ U^c : U \in \mathcal{K} \}.
\]

Let $C_\mathcal{K}(X)$ be the closure system on $X$ generated by $S(X)$, i.e., $C_\mathcal{K}(X) = \{ A : A \subseteq S(X) \}$.
The elements of $C_{\mathcal{K}(X)}$ are called \textit{subbasic closed subsets} of $X$. If  $C(X)$ is the collection of all closed subsets of the topology, observe that $S(X) \subseteq C_\mathcal{K}(X) \subseteq C(X)$. A subset $Y \subseteq X$ is called \textit{saturated} if it is an intersection of open sets. If $(X, \mathcal{K})$ is a topological space and $Y \subseteq X$, we say that a family $\mathcal{Z} \subseteq S(X)$ is a \textit{$Y$-family} if for every $A, B \in \mathcal{Z}$, there exist $H, C \in S(X)$ such that $Y \subseteq H$, $C \in \mathcal{Z}$, $A \cap H \subseteq C$, and $B \cap H \subseteq C$.
\\

By an $S$-space \cite{CG} we mean a topological space $\langle X, \mathcal{K}\rangle$ in which the following conditions hold:
\begin{itemize}
    \item[(S1)] $\langle X, \mathcal{K}\rangle$ is a $T_0$-space and $X = \bigcup \mathcal{K}$.
    \item[(S2)] $\mathcal{K}$ is a subbase of compact open subsets, it is closed under finite unions and $\emptyset \in \mathcal{K}$.
    \item[(S3)] For every $U, V \in \mathcal{K}$, if $x \in U \cap V$, then there exist $W, D \in \mathcal{K}$ such that $x \notin W$, $x \in D$ and $D \subseteq (U \cap V) \cup W$.
    \item[(S4)] If $Y \in C_{\mathcal{K}}(X)$ and $\mathcal{J} \subseteq S(X)$ is a $Y$-family such that $Y \cap A^c \neq \emptyset$ for all $A \in J$, then $Y \cap \bigcap \{A^c : A \in J\} \neq \emptyset$.
\end{itemize}

Let $\langle X, \mathcal{K}\rangle$ be an $S$-space. Then the structure $\mathbf{S}(X) = \langle S(X), \cap, X\rangle$ is a semilattice, called \emph{the dual semilattice of} $\langle X, \mathcal{K}\rangle$. Conversely, if $\mathbf{A}$ is a semilattice, then $\langle\mathcal{X}(\mathbf{A}), \mathcal{K}_\mathbf{A}\rangle$ is a topological space where $\mathcal{K}_\mathbf{A} = \{\beta(a)^c : a \in A\}$ is a subbase induced by $\beta$. Then it follows that $\beta : A \to S(\mathcal{X}(\mathbf{A}))$ is an isomorphism of semilattices and $\langle\mathcal{X}(\mathbf{A}), \mathcal{K}_\mathbf{A}\rangle$ is an $S$-space, called the \emph{dual $S$-space of $\mathbf{A}$}. If we consider the associated $S$-space $\langle\mathcal{X}(\mathbf{S}(X)), \mathcal{K}_{\mathbf{S}(X)}\rangle$ of the dual semilattice $\mathbf{S}(X)$, then the mapping $H_X : X \to \mathcal{X}(\mathbf{S}(X))$ given by $H_X(x) = \{A \in S(X) : x \in A\}$ for all $x \in X$, is a homeomorphism between $S$-spaces.
\\

Let $\mathsf{Rel}$ be the category of sets and binary relations, and $\mathsf{Set}$ the category of sets and functions. It is well known that there exists a faithful functor $\Box \colon \mathsf{Rel}^{op} \rightarrow \mathsf{Set}$ defined for each set $X$ by $\Box(X) = \mathcal{P}(X)$, and for each binary relation $T \subseteq X \times Y$ and each $U \in \mathcal{P}(Y)$, $\Box_T$ is defined by
\begin{equation*}
\Box_{T}(U) = \{ x \in X \colon T(x) \subseteq U \}.
\end{equation*}

A \emph{meet-relation} between two $S$-spaces $\langle X_{1}, \mathcal{K}_{1} \rangle$ and $\langle X_{2}, \mathcal{K}_{2} \rangle$ was defined in \cite{CG} as a relation $T \subseteq X_{1} \times X_{2}$ satisfying the following conditions:
\begin{enumerate}
\item $\Box_{T}(U) \in S(X_{1})$ for all $U \in S(X_{2})$,
\item $T(x) = \displaystyle \bigcap \{ U \in S(X_{2}) \colon T(x) \subseteq U \}$ for all $x \in X_{1}$.
\end{enumerate}

Furthermore, if $\langle X_{i},\mathcal{K}_{i}\rangle$, with $i=1,2,3$, are $S$-spaces and $R \subseteq X_{1} \times X_{2}$ and $T \subseteq X_{2} \times X_{3}$ are meet-relations, then their composition is defined by the following meet-relation:
\begin{equation}\label{Definicion composition meet-relations}
T \ast R = \{ (x,z) \in X_{1} \times X_{3} \colon \forall U\in S(X_{3}) [(T\circ R)(x)\subseteq U \Rightarrow z\in U ]\}.
\end{equation}

Thus, it turns out that $\ast$ is associative, and if $\sqsupseteq_{2}$ denotes the dual of the specialization order on $X_{2}$ ($x\sqsupseteq_{2} y$ if and only if $y$ belongs to the closure of $x$), then $T \ast \sqsupseteq_{2} = T$ and $\sqsupseteq_{2} \ast R = R$. That is, $S$-spaces and meet-relations form a category $\mathsf{Sspa}$, in which composition is represented by the operator $\ast$, and the identity arrow is determined by the dual of the specialization order. If we write $\mathsf{S}$ be the category of semilattices and homomorphisms. Then we can conclude:

\begin{theorem}[\cite{CG}]\label{Duality S y SSpa}
The categories $\mathsf{S}$ and $\mathsf{Sspa}$ are dually equivalent.
\end{theorem}

Given a semilattice $\mathbf{A}$, a map $m : A \to A$ is said to be \emph{monotone} if it preserves the underlying order of $\mathbf{A}$. Equivalently, $m$ is monotone if for every $a,b\in A$, $m(a\wedge b)\leq m(a)\wedge m(b)$. By a \emph{monotone semilattice} we mean a pair $\langle\mathbf{A}, m\rangle$ where $\mathbf{A}$ is a semilattice and $m$ is a monotone operation on $\mathbf{A}$. It is readily seen that the class of monotone semilattices is a variety. Let $\langle \mathbf{A} ,m\rangle$ and $\langle\mathbf{B}, n\rangle$ be two monotone semilattices.
 We say that a semilattice homomorphism $ h : \mathbf{A} \rightarrow \mathbf{B} $ is a \emph{monotone homomorphism} if $h(m(a)) = n(h(a))$ for all $a \in A$. We write $\mathsf{mMS}$ to denote both the category of monotone semilattices and monotone homomorphisms, as well as the underlying variety. 
\\

A \emph{multirelation on set $X$} is a subset of the cartesian product $X \times \mathcal{P}(X)$, that is, a set of ordered pairs $(x, Y)$ where $x \in X$ and $Y \subseteq X$. If $\langle X , \mathcal{K}\rangle$ is an S-space, $\mathcal{Z}(X)$ is the set of all subsets of $X$
that are intersections (possibly infinite) of members of $\mathcal{K}$.
\\

An $mS$-space is a structure $\langle X, \mathcal{K}, R \rangle$, where $\langle X, \mathcal{K} \rangle$ is an $S$-space and $R \subseteq X \times \mathcal{Z}(X)$ is a multirelation such that:
\begin{enumerate}
    \item $m_R(U) = \{ x \in X : R(x)\subseteq L_{U} \} \in S(X)$ for all $U \in S(X)$,
    \item $R(x) = \bigcap\{ L_U : U \in S(X) \text{ and } x \in m_R(U) \}$ for all $x \in X$.
\end{enumerate}

where, for every $U \in S(X)$, $L_{U} = \{ Z \in \mathcal{Z}(X) \colon Z \cap U \neq \emptyset \}$.
\\

Let $\langle X_1, \mathcal{K}_1, R_1 \rangle$ and $\langle X_2, \mathcal{K}_2, R_2 \rangle$ be two $mS$-spaces. We say that a meet-relation $T \subseteq X_1 \times X_2$ is a \emph{monotone meet-relation} if the following diagram commutes:

\[
\xymatrix{
S(X_2) \ar[r]^{\Box_T} \ar[d]_{m_{R_2}} & S(X_1) \ar[d]^{m_{R_1}} \\
S(X_2) \ar[r]_{\Box_T} & S(X_1)
}
\]
Let $\mathsf{mSp}$ be the category of $mS$-spaces and monotone meet-relations in which the identity map is given by the dual of the specialization order. In \cite{CMZ} it was proved the following:
\begin{theorem}\label{duality Msp mMS}
    The categories $\mathsf{mSp}$ and $\mathsf{mMS}$ are dually equivalent.
\end{theorem}
If $\langle \mathbf{A},m\rangle$ is a monotone semilattice, its associated $mS$-space is $\langle \mathcal{X}(\mathbf{A}),\mathcal{K}_\mathbf{A},R_m\rangle$. Converseley, if $\langle X,\mathcal{K},R\rangle$ is a $mS$-space, its associated monotone semilattice is $\langle \mathbf{S}(X),m_R\rangle$.

\section{Semilattices with an adjunction}\label{Semilattices with an adjunction}

Let $(P,\leq)$ and $(Q, \leq)$ be two ordered sets and consider two maps $f: P\rightarrow Q$ and $g: Q\rightarrow P$. We say that the pair $(f,g)$ is an \emph{adjoint pair} between $(P,\leq)$ and $(Q, \leq)$ (or simpliy adjoint pair), if for all $p \in P$ and  $q \in Q$, 
\[f(p)\leq q\quad \text{if and only if} \quad p \leq g(q).\]
In such a case, $f$ is called the \emph{left adjoint} and $g$ is called the \emph{right adjoint}, of the adjoint pair $(f,g)$. In what follows we write $f\dashv g$ to denote that $(f,g)$ is an adjoint pair between $(P,\leq)$ and $(Q, \leq)$. If the posets $(P, \leq)$ and $(Q, \leq)$ are identical, then $f \dashv g$ means that the pair $(f, g)$ is an adjoint pair on $P$. It is readily seen that $f$ and $g$ are monotone maps and it is well known (see chapter 7 of \cite{DP2022}) that $f$ preserves all the existing joins of $(P,\leq)$ and $g$ preserves all the existing meets of $(Q,\leq)$. We note that adjoint pairs of posets are also referred to as \emph{Galois connections} in the literature (see \cite{DP2022,6,Dzik10,8,9,10,ore1944} and all the references therein). However, we prefer the term adjoint pair of posets because this concept aligns with the notion of an adjoint pair in category theory when posets are viewed as categories.

\begin{definition}
Let $\mathbf{A}$ be a semilattice and let $i$ and $d$ be unary operators on $A$. A  semilattice with an adjunction (SLata, for short) is an algebra $\langle\mathbf{A},i,d\rangle$ such that $i\dashv d$.
\end{definition}

\begin{remark}\label{condicion adjuncion ecuacional}
If $\mathbf{A}$ is a semilattice then it follows that $i\dashv d$ if and only if \[i(d(a))\leq a \leq d(i(a)),\]
for every $a\in A$. Thus, the class of SLata forms a variety.
\end{remark}

Next, we provide some examples of SLatas:

\begin{example}
    Let $\mathbf{A}$ be a semilattice and consider the identity map $id_A$. Then, $\langle \mathbf{A}, id_{A},id_{A}\rangle$ is trivially an SLata.
\end{example}

\begin{example}
    Let $\mathbf{A}$ be a semilattice with a bottom element $0$. Then, consider $i,d:A\to A$, defined as $i(a)=0$ and $d(a)=1$. It is clear that $\langle \mathbf{A},i,d\rangle$ is an SLata. 
\end{example}

\begin{example}
Let $\mathbf{L}$ be a complete lattice and consider $f:L\to L$ to be an arbitrary monotone function. Now define $i,d:L\to L$ as follows:
\[i(a)=\bigwedge \{x\in L\colon x\leq f(a)\},\]
\[d(a)=\bigvee \{x\in L\colon f(a)\leq x\}.\]
Then, $\langle \mathbf{L},i,f\rangle$ and $\langle \mathbf{L},f,d\rangle$ are SLatas.
\end{example}

\begin{example}
    Let $S$ be a set and consider a function $f:S\to S$. It is well known that, for every $X,Y\subseteq S$, 
    \[f[X]\subseteq Y \Longleftrightarrow X\subseteq f^{-1}[Y].\] 
    Since $\mathbf{P}_{S}=\langle\mathcal{P}(S),\cap, S\rangle$ is a semilattice, then, $\langle \mathbf{P}_{S}, f,f^{-1}\rangle$ is an SLata. 
\end{example}

\begin{example}
   Let  \( X = \{a, b\} \) and consider the semilattice $\mathbf{P}_{X}$. Define $f,g:\mathcal{P}(X)\rightarrow \mathcal{P}(X)$ as follows:
\[
f(S) = S \cup \{b\}, \quad g(T) = T \setminus \{b\}.
\]
It easily follows that \( f \dashv g \), so $\langle \mathbf{P}_X,f,g\rangle$ is an SLata.
\end{example}

\begin{example}
    Let \( P = \mathbb{N} \cup \{\infty\} \), and consider the meet semilattice $\mathbf{A}=\langle P,\wedge,\top\rangle$ where \( a \wedge b = \min(a, b) \), with respect to the usual order, and \(\top = \infty\). Now, let us define $f.g:P\to P$ by:
\[
f(x) = x + 1, \quad \text{and}\quad  g(y) = \max(0, y - 1).
\]

We claim that \( f \dashv g \). Indeed, if \( f(x) \leq y \), then \( x + 1 \leq y \). This implies \( x \leq y - 1 \). Since \( g(y) = \max(0, y - 1) \), we have \( x \leq g(y) \). On the other hand, if \( x \leq g(y) \), then \( x \leq \max(0, y - 1) \). If \( y > 0 \), this implies \( x + 1 \leq y \), i.e., \( f(x) \leq y \). Thus, \( f \dashv g \), as claimed. Therefore, $\langle \mathbf{A},f,g\rangle$ is an SLata.
\end{example}

Given two semilattices \( \mathbf{A} \) and \( \mathbf{B} \) and an order-preserving map \( f : A \to B \), we recall that we can define define \( R_f \subseteq \mathcal{X}(\mathbf{B}) \times \mathcal{Z}(\mathcal{X}(\mathbf{A})) \) by
\[
(P, Z) \in R_f \iff f^{-1}[P] \cap I_A(Z) = \emptyset,
\]
where $I_\mathbf{A}(Z) = \{a \in A : \beta(a) \cap Z = \emptyset\}.$
\\

As a first step in constructing our categorical duality, we proceed to define the multi-relational topological spaces that will be used in subsequent sections:

\begin{definition}\label{Slata Space}
    A structure $\langle X,\mathcal{K},I,D \rangle $  is an SLata-space if the following conditions hold:
    \begin{enumerate}
    \item  $\langle X,\mathcal{K},I\rangle $, $\langle X,\mathcal{K},D\rangle $ are $mS$-spaces. 
    \item For every $U \in S(X)$, if $x\in U$ then for every $Z\in D(x)$  exists $w\in Z$ such that $I(w)\subseteq L_U$. 
    \item For every $U \in S(X)$, if for every $Z \in I(x)$ there exists $y \in Z$ such that  $D(y)\subseteq L_U$, then $x\in U$.      
\end{enumerate}

\end{definition}
\begin{proposition}\label{SLata-Sp to Slata}
A structure $\langle X,\mathcal{K},I,D\rangle $  is an SLata-space if and only if  $\langle \mathbf{S}(X),m_I,m_D\rangle $ is an SLata.
\end{proposition}
    \begin{proof}  Let us assume that $\langle X,\mathcal{K},I,D\rangle $ is an SLata-space. From Theorem \ref{duality Msp mMS}, both, $\langle \mathbf{S}(X),m_{I}\rangle$ and $\langle \mathbf{S}(X),m_{I}\rangle$ are monotone semilattices. We will proof that $ m_I \dashv m_D$. To do so, we will show that for every $U\in S(X)$, $m_I(m_D(U))\subseteq U\subseteq m_D(m_I(U))$. Let $U \in S(X)$. In order to see that $U\subseteq m_D(m_I(U))$, let   $x\in U$ and let $Z\in D(x)$. By Definition \ref{Slata Space} (2), there exists $w\in T$ such that $I(w) \subseteq L_U$. Thus, $Z\cap m_{D}(U)\neq \emptyset$ and consequently $Z\in  L_{m_{I}(U)} $. I.e. $D(x)\subseteq L_{m_{I}(U)}$. Hence, by definition of $m_{D}$ we get $x\in m_{I}(m_{D}(U))$. Now, in order to show the remaining inclusion, let $x\in m_{I}(m_{D}(U))$. Then, $I(x) \subseteq L_{m_{D}(U)}$. Let $Z\in I(x)$. Then there exists $y\in Z $ such that $D(y) \subseteq L_U$. Therefore, from Definition \ref{Slata Space} (3), $x\in U$. I.e. $m_I(m_D(U))\subseteq U$. On the other hand, let us assume that $\langle \mathbf{S}(X),m_{I},m_{D}\rangle $ is an SLata. Then, for every $U\in S(X)$, $m_I(m_D(U))\subseteq U\subseteq m_D(m_I(U))$. It is not hard to see that the latter inclusions implies conditions (2) and (3) of Definition \ref{Slata Space}. We leave the details to the reader. This concludes the proof.

    \end{proof}
\begin{proposition}\label{SLata to Slata-Sp}
Let $\mathbf{A}$ be a semilattice. An algebra $\langle \mathbf{A},i,d \rangle $ is an SLata if and only if $ \langle \mathcal{X}(\mathbf{A}),\mathcal{K}_\mathbf{A}, R_{i}, R_{d} \rangle$ is an SLata-space.    
\end{proposition}
\begin{proof}
 For the only if part, let  $\langle \mathbf{A},i,d \rangle$ be a SLata. From Theorem \ref{duality Msp mMS}, we have that both $\langle \mathcal{X}(\mathbf{A}),\mathcal{K}_\mathbf{A}, R_{i}\rangle$ and $\langle \mathcal{X}(\mathbf{A}),\mathcal{K}_\mathbf{A}, R_{d}\rangle$ are $mS$-spaces. Now, in light of Proposition \ref{SLata-Sp to Slata}, to prove the claim it is enough to show that $m_{R_{i}} \dashv m_{R_{d}}$. So let $U\in S(\mathcal{X}(\mathbf{A}))$. We will check that $U\subseteq m_{R_{d}}(m_{R_{i}}(U))$. Indeed, since the map $\beta : A \rightarrow S(\mathcal{X}(\mathbf{A}))$ is an isomorphism of semilattices, we can write $U =\beta(a)$, for some $a\in A$. By assumption and Remark \ref{condicion adjuncion ecuacional}, we know that $a\leq d(i(a))$. So, it follows that $\beta(a)\subseteq \beta (d(i(a)))$. Then, from  lemmas 7 and 10 of \cite{CMZ}, we obtain $\beta (d(i(a)))= m_{R_{d}}(m_{R_{i}}(\beta(a)))$. Therefore, $\beta(a)\subseteq m_{R_{d}}(m_{R_{i}}(\beta(a)))$. The proof of the remaining inclusion is similar. For the if part, let us assume that $\langle\mathcal{X}(\mathbf{A}),\mathcal{K}_\mathbf{A},R_{i}, R_{d}\rangle$ is an SLata-space. From Theorem \ref{duality Msp mMS} and Proposition \ref{SLata-Sp to Slata}, we get that $\langle \mathbf{S}(\mathcal{X}(\mathbf{A})),m_{R_{i}},m_{R_{d}}\rangle$ is a SLata and since $\beta : A \rightarrow S(X(A))$ is a semilattice isomorphism, from Remark \ref{condicion adjuncion ecuacional} we get that $m_{R_{i}}(m_{R_{d}}(\beta(a)))\subseteq \beta (a)$, for all $a\in A $. Let $a\in A$, then, again from  lemmas 7 and 10 of \cite{CMZ} we obtain  $\beta(i(d(a)))\subseteq\beta(a)$.
Finally, since $\beta$ is in particular an order isomorphism, then it follows that $i(d(a)))\leq a$. The proof of the remaining inequality is analogous.

\end{proof}
\begin{definition}
    Let $\langle \mathbf{A}_1,i_1,d_1\rangle$ and $\langle \mathbf{A}_2,i_2,d_2\rangle$ be two SLata's. Then a map $h:A_1\rightarrow A_2$ is a SLata morphism if the following conditions hold:
    \begin{itemize}
        \item[1)] $h$ is a morphism of semilattices.
        \item[2)] $h({i}_1(a))={i}_2(h(a))$, for all $a\in A_1$.
        \item[3)] $h({d}_1(a))={d}_2(h(a))$, for all $a\in A_1$.      
    \end{itemize}
\end{definition}

We denote by $\mathsf{SLata}$ both the variety of SLata's and the category whose objects are SLata's and whose arrows are SLata morphisms. The composition of arrows is defined as the usual composition of functions, and the identity arrow is given by the identity map.

\begin{definition}

    Let $\langle X_1,\mathcal{K}_1,I_1,D_1\rangle $ and $\langle X_2,\mathcal{K}_2,I_2,D_2\rangle$ two SLata-spaces. A meet- relation $T\subseteq X_1\times X_2$ is a SLata-relation if  the following diagrams commute:
    \[
\begin{tikzcd}
S(X_2) \arrow[r, "\Box_T"] \arrow[d, "m_{I_2}"'] & S(X_1) \arrow[d, "m_{I_1}"] \\
S(X_2) \arrow[r, "\Box_T"'] & S(X_1)
\end{tikzcd}\hspace{2cm}
\begin{tikzcd}
S(X_2) \arrow[r, "\Box_T"] \arrow[d, "m_{{D}_2}"'] & S(X_1) \arrow[d, "m_{D_1}"] \\
S(X_2) \arrow[r, "\Box_T"'] & S(X_1)
\end{tikzcd}
\]
\end{definition}
\

Let $\langle X_j, \mathcal{K}_{j},I_j,D_j\rangle$ with $j = 1, 2, 3$ be SLata-spaces and $H \subseteq X_1 \times X_2$ and $T \subseteq X_2 \times X_3$ be SLata-relations. It is clear that the composition between $H$ and $T$ defined as $T \circ H := T\ast R$ is well defined and associative. Moreover, it is readily seen that for every SLata-space $\langle X, \mathcal{K}, I, D \rangle$, the dual of the specialization order $\sqsupseteq \subseteq X \times X$ is a SLata meet-relation that acts as the identity arrow with respect to the composition previously defined. Therefore, this data allows us to define a category whose objects are SLata-spaces and whose arrows are SLata-relations. We denote this category by $\mathsf{SLataSp}$.

\begin{proposition}\label{unit duality}
  Let  $\langle\mathbf{A},i,d\rangle$ be an SLata, then $\beta: \mathbf{A} \rightarrow \mathbf{S}(\mathcal{X}(\mathbf{A}))$ is an isomorphism of SLata's.
\end{proposition}
\begin{proof}
 From Theorem \ref{duality Msp mMS}, we know that $\beta$ establishes an isomorphism of monotone semilattices between $\langle \mathbf{A}, i \rangle$ and $\langle \mathbf{S}(\mathcal{X}(\mathbf{A})), m_I \rangle$, as well as between $\langle \mathbf{A}, d \rangle$ and $\langle \mathbf{S}(\mathcal{X}(\mathbf{A})), m_D \rangle$. It only remains to show that $\beta$ preserves $i$ and $d$. This is indeed the case, since from Lemmas 7 and 10 of \cite{CMZ}, we have that for all $a \in A$, $m_I(\beta(a)) = \beta(i(a))$ and $m_D(\beta(a)) = \beta(d(a))$.
\end{proof}

\

Let $\langle \mathbf{A}_j, i_j, d_j \rangle$, with $j = 1, 2$, be two SLata's, and let $h : \mathbf{A}_1 \to \mathbf{A}_2$ be a semilattice homomorphism. From Lemma 11 of \cite{CMZ}, it follows that $h$ is a monotone homomorphism (with respect to $i_j$ and $d_j$, respectively) if and only if $R_h \subseteq \mathcal{X}(\mathbf{A}_2) \times \mathcal{X}(\mathbf{A}_1)$, defined by $(P, Q) \in R_h \Leftrightarrow h^{-1}[P] \subseteq Q$, is a monotone meet-relation between their respective dual $mS$-spaces. In particular, if $h$ is an SLata morphism, then $R_h$ is a SLata-relation. Furthermore, if $h = \text{id}_A$, we have $R_h = \sqsupseteq_{\mathcal{X}(\mathbf{A})}$, and from Lemma 12 of \cite{CMZ}, we also have that if $h : A \to B$ and $g : B \to C$ are morphisms of SLata, then $R_{gh} = R_h \ast R_g$.

As a result of the latter discussion, we can consider the following assignments:

\[
\begin{array}{rcl}
   \langle\mathbf{A}, i, d\rangle & \mapsto & \langle \mathcal{X}(\mathbf{A}), \mathcal{K}_\mathbf{A}, R_i, R_d \rangle, \\
   h & \mapsto & R_h, 
\end{array}
\]

so we have a functor \(G : \mathsf{SLata}^{op} \to \mathsf{SLataSp}\).

\

We recall that in \cite{CMZ}, it was proved that the map  $H_{X} : X \rightarrow \mathcal{X}(\mathbf{S}(X))$ defined by
$H_{X}(x) = \{U \in S(X) : x \in U\}$
induces an isomorphism of $mS$-spaces. Taking this into account, we obtain:
 
\begin{proposition}\label{counit duality}
    Let $\langle X,K,I,D\rangle $ an SLata-space. Then $T_{X}\subseteq \mathcal{X}(\mathbf{S}(X)) \times X$ defined by 
    \begin{displaymath}
        \begin{array}{ccc}
            (H_{X}(y), x) \in T_X & \Leftrightarrow  & H_{X}(y)\subseteq H_{X}(x)
        \end{array}
    \end{displaymath}
 is an isomorphism of SLata-spaces. 
\end{proposition}

Moreover, we are also able to define a functor $F:\mathsf{SLataSp} \rightarrow \mathsf{SLata}^{op} $, where

\begin{displaymath}
\begin{array}{rcl}
    \langle X,\mathcal{K},I,D\rangle  & \mapsto & \langle \mathbf{S}(X),m_{I},m_{D}\rangle \\
     T & \mapsto & \Box_{T}. 
\end{array}
\end{displaymath}
  
From Propositions \ref{unit duality} and \ref{counit duality}, it follows immediately that:

\begin{theorem}\label{duality SLata}
    The categories $\mathsf{SLataSp}$ and $\mathsf{SLata}$ are dually equivalent.
\end{theorem}

The following result, which states that certain bijective functions between SLata-spaces induce isomorphisms between them, is analogous to Lemma 8 of \cite{CMZ}. Since its proof is similar, we omit it. 

\begin{lemma}
    Let $\langle X_1,\mathcal{K}_{1},I_1,D_1\rangle$ and $\langle X_2,\mathcal{K}_{2},I_2,D_2\rangle$ be two SLata-spaces and let $h:X_1\to X_2$ be a bijective function such that:
    \begin{enumerate}
        \item $ \mathcal{K}_{2} = \{ h[V] : V \in \mathcal{K}_{1} \} $,
        \item For all $x\in X_1$ and $Z\in \mathcal{Z}(X_1)$, $(x,Z)\in I_1\Longleftrightarrow (h(x),h[Z])\in I_2$.
        \item For all $x\in X_1$ and $Z\in \mathcal{Z}(X_1)$, $(x,Z)\in D_1\Longleftrightarrow (h(x),h[Z])\in D_2$.
    \end{enumerate}
    Then, these SLata-spaces are isomorphic, and moreover, such an isomorphism is induced by $f$.
\end{lemma}

We conclude this section with a result stating that bijective maps preserving subbasic elements between SLata-spaces and $mS$-spaces induce extensions of isomorphisms of SLata-spaces.

\begin{lemma}\label{Analogo Lema 8}
    Let $\langle X_1,\mathcal{K}_{1},I,D\rangle$ an SLata-space and $\langle X_{2},\mathcal{K}_{2},T\rangle$ an mS-space. If $ h : X_1 \to X_2 $ is a bijective map such that:
    \begin{enumerate}
        \item $ \mathcal{K}_{2} = \{ h[V] : V \in \mathcal{K}_{1} \} $,
        \item $(x,Z)\in I$ if and only if $(h(x),h[Z])\in T$, for all $x\in X_1$ and $Z\in \mathcal{Z}(X_1)$.
    \end{enumerate}
     Then, there exists a multirelation $ T_2 \subseteq X_2 \times \mathcal{Z}(X_2) $ such that $ \langle X_1, \mathcal{K}_1, I,D\rangle$ and $ \langle X_2, \mathcal{K}_2, T, T_{2}\rangle $ are isomorphic SLata-spaces.
\end{lemma}
    \begin{proof}
      Notice that from the assumptions on $h$ and Lemma 8 of \cite{CMZ}, we obtain that $ \langle X_2, \mathcal{K}_2, T \rangle$ is an $mS$-space isomorphic to $\langle X_1,\mathcal{K}_{1},I\rangle$.
     Now, we define the multirelation $T_{2} \subseteq X_{2} \times  \mathcal{Z}(X_{2})$ as follows:
      \begin{center}
          $(h(x),W) \in T_{2} \Longleftrightarrow (x, h^{-1}[W]) \in D.$
      \end{center}
In order to prove that $\langle X_2,\mathcal{K}_2,T,T_2\rangle$ is a SLata-space, observe that since $\langle X_1,\mathcal{K}_1,D\rangle$ is an $mS$-space, from the definition of $T_2$ and the assumptions on $h$, by applying Corollary 1 of \cite{CMZ} it follows that $\langle X_2,\mathcal{K}_2,T_2\rangle$ is an $mS$-space. Now we proceed to verify that condition (2) of Definition \ref{Slata Space} holds. I.e. we must show that for every $U \in S(X_{2})$, if $h(x)\in U$, then, for every $W \in T_{2}(h(x))$ there exists $u\in W$ such that $T(u)\subseteq L_U$. To do so, let \( U \in S(X_{2}) \) and assume \( h(x) \in U \). From (1), it is immediate that \( U = h[V^{c}] \) for some \( V \in \mathcal{K}_{1} \). Thus, \( x \in V^{c} \), since \( h \) is bijective. Let \( W \in T_{2}(h(x)) \). By the definition of \( T_{2} \), we have that \( h^{-1}[W] \in D(x) \). Since \( \langle X_1, \mathcal{K}_1, I, D \rangle \) is an SLata-space and \( V^c \in \mathcal{K}_1 \), there exists \( w \in h^{-1}[W] \) such that \( I(w) \subseteq L_{V^c} \). Thus, \( h(w) \in W \), and by the bijectivity of \( h \), we have \( h[I(w)] \subseteq h[L_{V^c}] = L_{U} \). Observe that, by (2) and Lemma \ref{Analogo Lema 8}, it is readily seen that \( h[I(w)] = T(h(w)) \). Hence, if we take \( u = h(w) \), it is clear that condition (2) of Definition \ref{Slata Space} holds, as desired. Now, we recall that to prove (3) of Definition \ref{Slata Space}, we need to show that for every $U\in S(X_2)$, if for every $W\in T(h(x))$ there exists $u\in W$ such that $T_2(u)\subseteq L_U$, then $h(x)\in U$. At this stage it is clear that we can set $U=h[V^c]$, for some $V\in \mathcal{K}_1$. So let us consider $W\in T(h(x))$ such that there exists $u\in W$ with $T_2(u)\subseteq L_U$. From (2) and the bijectivity of $h$ we get that $h^{-1}[W]\in I(x)$ and $h^{-1}(u)\in h^{-1}[W]$. So, by definition of $T_2$, it is clear that $T_2(u)=h[D(h^{-1}(u))]$. Therefore, we can conclude that $D(h^{-1}(u))\subseteq L_{V^c}$. Since \( \langle X_1, \mathcal{K}_1, I, D \rangle \) is an SLata-space, then we get that $x\in V^c$, so $h(x)\in U$. Therefore, (2) of Definition \ref{Slata Space} holds, as granted. So we can conclude that \( \langle X_2, \mathcal{K}_2, T, T_2 \rangle \) is an SLata-space.

Finally, from (2), the Definition of $T_2$ and Lemma \ref{Analogo Lema 8} we conclude that $\langle X_1,\mathcal{K}_1,I,D\rangle$ and $\langle X_2,\mathcal{K}_2,T,T_2\rangle$ are isomorphic SLata-spaces. This concludes the proof.



\end{proof}

\section{Congruences and SLata-Vietoris Families}\label{Congruences and SLata-Vietoris Families}

In \cite{CMZ}, the authors characterize congruences of semilattices and monotone semilattices using lower-Vietoris-type families and monotone lower-Vietoris-type families associated with their dual \(mS\)-spaces. The aim of this section is to extend this characterization to the case of SLata's and their corresponding dual SLata-spaces. Taking into account that, in the context of varieties, congruences and surjective homomorphisms are equivalent, we begin by studying surjective homomorphisms in \(\mathsf{SLata}\). This leads us to introduce the notion of an SLata-Vietoris family for an SLata-space, which naturally generalizes the concept of monotone lower-Vietoris-type families for \(mS\)-spaces. Subsequently, we will examine how, given an SLata-Vietoris family, one can construct a one-to-one SLata-relation, thus establishing a connection with congruences.

\subsection{One-to-one SLata-relations}\label{One-to-one SLata-relations} 

We begin by recalling some preliminaries. If \(\langle X_1, \mathcal{K}_1\rangle\) and \(\langle X_2, \mathcal{K}_2\rangle\) are two \(S\)-spaces, a meet-relation \(T \subseteq X_1 \times X_2\) is \emph{one-to-one} if, for each \(x \in X_1\) and \(U \in S(X_1)\) with \(x \notin U\), there exists \(V \in S(X_2)\) such that \(U \subseteq \Box_{T}(V)\) and \(x \notin \Box_{T}(V)\). Similarly, a monotone meet-relation between $mS$-spaces is one-to-one, if it is one-to-one as a meet relation. In \cite{CMZ}, it was shown that one-to-one meet relations and monotone one-to-one meet relations correspond precisely to the extremal monomorphisms in $\mathsf{Sspa}$ and \(\mathsf{mSp}\), respectively, which, in turn, correspond to onto homomorphisms (extremal epimorphisms) in $\mathbf{S}$ and \(\mathsf{mMS}\), repectively.

\begin{definition}\label{one-to-one SLata-relation}
    Let $ \langle X_1, \mathcal{K}_1 , I_{1}, D_{1}\rangle$ and $\langle X_2, \mathcal{K}_2,I_{2},D_{2} \rangle$ be two SLata-spaces. An SLata-relation $T \subseteq X_1 \times X_2$ is
one-to-one if it is one-to-one as a meet-relation.

\end{definition}

The following results are the SLata-space counterparts of Theorem 15 and Corollary 4 from \cite{CMZ}. Since the proofs are analogous, they are omitted.

\begin{theorem}\label{SLata onto}
   Let $ \langle X_1, \mathcal{K}_1 , I_{1}, D_{1}\rangle$ and $\langle X_2, \mathcal{K}_2,I_{2},D_{2} \rangle$ be two SLata-spaces  and $T \subseteq X_1 \times X_2$ be an SLata-relation.
Then the map $ \Box_{T} : S(X_2) \to S(X_1)$ is onto if and only if $T$ is a one-to-one SLata-relation.

\end{theorem}
\begin{cor}\label{SLata one-to-one}
    Let $\langle \mathbf{A}_1,i_1,d_1\rangle$ and $\langle \mathbf{A}_2,i_2,d_2\rangle$ be two SLata and let $ h: A \rightarrow B $ be a morphism in $\mathsf{SLata}$. Then $h$ is
onto if and only if the SLata-relation $R_{h}$ is one-to-one.
\end{cor}
    
Let $\langle \mathbf{A},i,d\rangle$ be an SLata, $\langle X, \mathcal{K},I, D \rangle$ be an SLata-space, and $T \subseteq X \times \mathcal{X}(\textbf{A}) $ be a one to one SLata-relation. In what follows, we consider the families $\mathcal{F}_T = \{T(x): x \in X \}\subseteq C_\mathcal{K}(\mathcal{X}(\mathbf{A}))$, and $\mathcal{M}=\{H_a\colon a\in A\}$, where for every $a \in A$, 
\[H_a = \{T(x): T(x) \cap \beta(a)^{c} \not= \emptyset\}.\] 

\begin{remark}\label{Remark util}
    For every $a \in A$ and $x \in X$ the following holds:
\[ T(x) \in H_a \iff x \in (\Box_{T}(\beta(a)))^c. \]
\end{remark}

If $\mathbf{A}$ and $\mathbf{B}$ are two semilattices and $h : \mathbf{A} \to \mathbf{B}$ is an onto homomorphism, then we can define a map $\lambda: \mathcal{X}(\mathbf{B}) \to \mathcal{F}_{{R}_h}$ by  
\[
\lambda(P) = R_h(P).
\] 
By Corollary 3 of \cite{CMZ}, it follows that $\langle \mathcal{F_R}_h, \mathcal{M} \rangle$ is an $S$-space which is homeomorphic to $\langle \mathcal{X}(\mathbf{B}), \mathcal{K}_\mathbf{B} \rangle$.

\begin{theorem}\label{SLata Vietoris iso}
    
Let $ \langle \mathbf{A}_1,i_1,d_{1} \rangle $  and $\langle \mathbf{A}_2, i_{2},d_{2}\rangle $ be two SLata's, and $ h : A_1 \to A_2 $ be an onto homomorphism. Let $ I_{h},D_{h} \subseteq \mathcal{F}_{R_h} \times \mathcal{Z}(\mathcal{F}_{R_h}) $ be the relations defined by:

\begin{displaymath}
\begin{array}{ccc}
   (R_{h}(Q), Z) \in I_h  & \Leftrightarrow &  (Q, \lambda^{-1}[Z]) \in R_{i_2},\\
   \\
   (R_{h}(Q), Z) \in D_h  & \Leftrightarrow & (Q, \lambda^{-1}[Z]) \in R_{d_2}, 
\end{array}     
\end{displaymath}

where $ Q \in \mathcal{X}(\mathbf{B}) $ and $ Z \in \mathcal{Z}(\mathcal{F}_{R_h}) $. Then $\langle \mathcal{F}_{R_h}, \mathcal{M},I_h, D_h \rangle$ is an SLata-space which is isomorphic to $\langle \mathcal{X}(\mathbf{B}), \mathcal{K}_{\mathbf{B}},R_{i_2},R_{d_2} \rangle $ and satisfies the following conditions:

\begin{displaymath}
    \begin{array}{ccc}
         \beta(a)^{c} \in R_{i_1}[{R_{h}}(P)] & \Leftrightarrow & H_a \in I_{h} ({R_{h}}(P)),  \\
         \\
         \beta(a)^{c} \in R_{d_1}[{R_{h}}(P)] & \Leftrightarrow & H_a \in D_{h} ({R_{h}}(P)). 
    \end{array}
\end{displaymath}
for all $ a \in A $ and $ P \in \mathcal{X}(\mathbf{B}) $.

\end{theorem}
\begin{proof}
Notice that from Corollary 3 of \cite{CMZ}, \( \lambda \) is bijective, and by Theorem 16 of \cite{CMZ}, we have \( \mathcal{M} = \{ \lambda(V) : V \in \mathcal{K}_{\mathbf{B}} \} \). Furthermore, since \( \langle \mathcal{X}(\mathbf{B}), \mathcal{K}_{\mathbf{B}}, R_{i_2}, R_{d_2} \rangle \) is an SLata-space, and again by Theorem 16 of \cite{CMZ}, we conclude that \( \langle \mathcal{F}_{R_h}, \mathcal{M}, I_h \rangle \) is an \( mS \)-space that is isomorphic to \( \langle \mathcal{X}(\mathbf{B}), \mathcal{K}_{\mathbf{B}}, R_{i_2} \rangle \). Therefore, from Lemma \ref{Analogo Lema 8}, we obtain that \( \langle \mathcal{F}_{R_h}, \mathcal{M}, I_h, D_h \rangle \) is an SLata-space that is isomorphic to \( \langle \mathcal{X}(\mathbf{B}), \mathcal{K}_{\mathbf{B}}, R_{i_2}, R_{d_2} \rangle \). Finally, once again, from Theorem 16 of \cite{CMZ}, it follows that:
\[
\beta(a)^{c} \in R_{i_1}[ R_h(P) ] \Leftrightarrow H_a \in I_h( R_h(P)),
\]
and
\[
\beta(a)^{c} \in R_{d_1}[ R_h(P) ] \Leftrightarrow H_a \in D_h( R_h(P))
\]
for all \( a \in A \) and \( P \in \mathcal{X}(\mathbf{B}) \). 
\end{proof}

\subsection{SLata-Vietoris families}\label{SLata-Vietoris families}

Let $\langle X,\mathcal{K}\rangle$ be an S-space and \(\mathcal{F}\) a non-empty subset of \(C_\mathcal{K}(X)\). Let us consider \(\mathcal{M}_{\mathcal{F}} = \{U^{-}_{\mathcal{F}} : U \in \mathcal{K}\}\), where for each \(U \subseteq X\),
\[
U^{-}_{\mathcal{F}} = \{Y \in \mathcal{F} : Y \cap U \neq \emptyset\}.
\]
We say that $\mathcal{F}$ is a \emph{lower-Vietoris-type-family of $\langle X,\mathcal{K}\rangle$}, if the pair $\langle \mathcal{F}, \mathcal{M}_{\mathcal{F}}\rangle$ is an S-space, with $\mathcal{M}_{\mathcal{F}} = \{U^{-}_{\mathcal{F}} : U \in \mathcal{K}\}$.

\ 

Let \( \mathbf{A} \) be a semilattice. If \( \mathcal{F} \subseteq C_{\mathcal{K}_{\mathbf{A}}}(\mathcal{X}(\mathbf{A})) \) is a lower-Vietoris-type family of the dual 
\( S \)-space \( \langle\mathcal{X}(\mathbf{A}), \mathcal{K}_{\mathbf{A}}\rangle \), then we can define a relation \( T_{\mathcal{F}} \subseteq \mathcal{F} \times \mathcal{X}(\mathbf{A}) \) by:
\[
(Y, P) \in T_\mathcal{F} \Leftrightarrow P \in Y.
\]
From Lemma 16 of \cite{CMZ}, it follows that $T_\mathcal{F}$ is a one-to-one meet relation, $\mathcal{F}=\mathcal{F}_{T_\mathcal{F}}$ and moreover,

\[\mathcal{M}_{\mathcal{F}}=\{\beta(a)_{\mathcal{F}}^{-}\colon \beta(a)^c\in \mathcal{K}_{\mathbf{A}}\}=\{H_a\colon a\in A\}=\mathcal{M}.\]

It is worth mentioning that, by Theorem \ref{Duality S y SSpa}, it also holds that $\langle S(\mathcal{F}), \cap, \mathcal{F}\rangle$ is a semilattice, and $S(\mathcal{F})=\{H_a^c\colon a\in A\}$.

\

Let \(\langle X, \mathcal{K}, R \rangle\) be an \(mS\)-space and \(\mathcal{F}\) a lower-Vietoris-type family of \(\langle X, \mathcal{K} \rangle\). For any \(H \subseteq \mathcal{Z}(X)\), let us consider the set
\[
[H \cap \mathcal{K})_{\mathcal{F}} = \{V \in \mathcal{K} : \exists U \in H \cap \mathcal{K} \, [U_{\mathcal{F}}^{-} \subseteq V_{\mathcal{F}}^{-}]\}.
\]
A subset \(H \subseteq \mathcal{Z}(X)\) is said to be \emph{\(\mathcal{M}_{\mathcal{F}}\)-increasing} if 
\[
[H \cap \mathcal{K})_{\mathcal{F}} = H \cap \mathcal{K}.
\]
Moreover, a family \(\mathcal{F}\) of non-empty subbasic closed subsets of \(\langle X, \mathcal{K} \rangle\) is a \emph{monotone lower-Vietoris-type family} if \(\langle \mathcal{F}, \mathcal{M}_{\mathcal{F}} \rangle\) is an \(S\)-space and for every \(Y \in \mathcal{F}\) we have that 
\[
R[Y] = \{Z \in \mathcal{Z}(X) : \exists y \in Y \, [(y, Z) \in R]\}
\]
is an \(\mathcal{M}_{\mathcal{F}}\)-increasing subset.

\begin{definition}\label{SLata-Vietoris-type}
    Let \(\langle X, \mathcal{K}, I,D\rangle \) be a SLata-space. A family \( \mathcal{F} \) of non-empty subbasic closed
subsets of \( \langle X, \mathcal{K} \rangle \) is an SLata-lower-Vietoris-type family (SLata-Vietoris family, for short) if the following conditions hold:
\begin{enumerate}
    \item \( \langle \mathcal{F}, \mathcal{M}_{\mathcal{F}} \rangle \) is an \( S \)-space.
    \item For
every \( Y \in \mathcal{F} \), the sets
\[
I[Y] = \{ Z \in \mathcal{Z}(X) : \exists y \in Y \, [(y, Z) \in I] \},
\] and
\[
D[Y] = \{ Z \in \mathcal{Z}(X) : \exists s \in S \, [(s, Z) \in D] \},
\]
are  \( \mathcal{M}_{\mathcal{F}} \)-increasing subsets.
\item \( \langle \mathcal{F}, \mathcal{M}_{\mathcal{F}}, \widehat{I}, \widehat{D} \rangle \) is an \( SLata \)-space, where $\widehat{I}, \widehat{D}\subseteq \mathcal{F}\times \mathcal{Z}(\mathcal{F})$ are defined as follows:
\[
(Y,Z)\in \widehat{I} \quad \Leftrightarrow \quad Z\in \bigcap \{{L_{{U_{\mathcal{F}}^{-}}^{c}}}\colon U\in \mathcal{K}\cap I[Y]^{c}\},
\]
\[
(Y,Z)\in \widehat{D} \quad \Leftrightarrow \quad Z\in \bigcap \{{L_{{U_{\mathcal{F}}^{-}}^{c}}}\colon U\in \mathcal{K}\cap D[Y]^{c}\}.
\]
\end{enumerate}

\end{definition}


\begin{proposition}

    Let \( \langle \mathbf{A}, i,d\rangle \) be an SLata and \( \mathcal{F} \subseteq \mathcal{C_{K_{\mathbf{A}}}}(\mathcal{X}(\mathbf{A})) \) be an SLata-Vietoris- family of \( \langle \mathcal{X}(\mathbf{A}),\mathcal{K}_{\mathbf{A}},R_i,R_d\rangle \). Let us consider the relations \( I_{\mathcal{F}}, D_{\mathcal{F}} \subseteq \mathcal{F} \times \mathcal{Z}(\mathcal{F}) \) defined by

\[
(Y, Z) \in I_{\mathcal{F}} \iff Z \in \bigcap\{L_{H_{a}^{c}} : a \in A \text{ and } \beta(a)^{c} \notin R_i[Y]\}.
\]
\[
(Y, Z) \in D_{\mathcal{F}} \iff Z \in \bigcap\{L_{H_{a}^{c}} : a \in A \text{ and } \beta(a)^{c} \notin R_d[Y]\}.
\]

Then \( \langle \mathcal{F}, \mathcal{M}, I_{\mathcal{F}}, D_{\mathcal{F}} \rangle \) is an SLata-space and \( T_{\mathcal{F}} \subseteq \mathcal{F} \times \mathcal{X}(\mathbf{A}) \) is a one-to-one SLata-relation.
\end{proposition}
\begin{proof}
The result follows immediately from Lemma 17 of \cite{CMZ}, Definition \ref{SLata-Vietoris-type}, and the fact that $\widehat{I}=I_{\mathcal{F}}$ and $\widehat{D}=D_{\mathcal{F}}$.
\end{proof}

\begin{proposition}\label{H preserva el orden}
    Let $\langle\mathbf{A},i,d\rangle$ be a SLata, $\langle X,\mathcal{K}\rangle$ be a S-space and let $T\subseteq X\times \mathcal{X}(\mathbf{A})$ be a one to one SLata-relation. If $a\leq b$, then $H_b\subseteq H_a$.
\end{proposition}
\begin{proof}
    Let $a, b \in A$ with $a\leq b$, then $\beta(a)\subseteq \beta(b)$. So, $\Box_{T}(\beta(a))\subseteq \Box_{T}(\beta(b))$. Let $T(x)\in H_{b}$. Then, from Remark \ref{Remark util}, we have that $ x\in (\Box_{T}(\beta (b)))^{c}$, thus $T(x) \not\subseteq \beta(b)$. So, $T(x) \not\subseteq \beta(a)$, and again from  Remark \ref{Remark util}, we obtain $T(x)\in H_a$, as desired.
    
\end{proof}

\begin{theorem}
    
Let $\langle \mathbf{A},i,d \rangle\ $ be an SLata and $\langle X, \mathcal{K}, I,D\rangle $ be an \( SLata \) space. If \( T \subseteq X \times \mathcal{X}(\mathbf{A}) \) is a
one-to-one SL- relation, then \( \mathcal{F}_T \) is a SLata-Vietoris-family of
\(\langle \mathcal{X}(\mathbf{A}),\mathcal{K}_{\mathbf{A}}, R_{i},R_{d} \rangle\).
\end{theorem}
\begin{proof}
At this point, it is clear that $\langle\mathcal{F}_{T},\mathcal{M}\rangle$ is an S-space. Now, from Theorem 17 of \cite{CMZ}, it follows that $R_{i}[T(x)]$ and $R_{d}[T(x)]$ are $\mathcal{M}_{\mathcal{F}_T}$-increasing. So, in order to prove our claim, by Definition \ref{SLata-Vietoris-type}, it only remains to show that $\langle \mathcal{F}_T, \mathcal{M}, I_{\mathcal{F}_T}, D_{\mathcal{F}_T}\rangle$ is an SLata-space. To this end, we will show first that $\langle S(\mathcal{F}_T),m_{I_{\mathcal{F}_T}},m_{D_{\mathcal{F}_T}} \rangle$ is an SLata. Thus, we need to proof that for every $a\in A$: 
\begin{equation}\label{ecuacion demostracion}
    m_{I_{\mathcal{F}_T}}(m_{D_{\mathcal{F}_T}}(H_a^c))\subseteq H_a^c\subseteq m_{D_{\mathcal{F}_T}}(m_{I_{\mathcal{F}_T}}(H_a^c)).
\end{equation}
Indeed, from Remark \ref{condicion adjuncion ecuacional}, for all $a\in A$, we have that 
\[i(d(a))\leq a\leq d(i(a)).\]

Then, as a consequence of Proposition \ref{H preserva el orden} and Lemma 17 of \cite{CMZ}, we obtain \[m_{I_{\mathcal{F}_T}}(m_{D_{\mathcal{F}_T}}(H_a^c))=H_{i(d(a))}^c\subseteq H_{a}^c\subseteq H_{d(i(a))}^c=m_{D_{\mathcal{F}_T}}(m_{I_{\mathcal{F}_T}}(H_a^c)).\]
Therefore, (\ref{ecuacion demostracion}) holds. Finally, from Proposition \ref{SLata-Sp to Slata} we have that $\langle\mathcal{F}_{T},\mathcal{M},\widehat{I},\widehat{D} \rangle$ is an SLata-space. Hence, $\mathcal{F}_{T}$ is an SLata-Vietoris family of the dual SLata-space of $\langle \mathbf{A},i,d\rangle$, as desired.

\end{proof}

\begin{theorem}\label{hom sobre a SLata-Vietoris}
    Let \(\langle \mathbf{A}_1,i_{1},d_{1} \rangle\) and \(\langle \mathbf{A}_2, i_{2},d_{2}\rangle \) be two SLata's, and let \( h : A_1 \to A_2 \) be an onto SLata homomorphism.  Then, \( \mathcal{F}_{R_h} \) is a SLata-Vietoris family of \(\langle \mathcal{X}(\mathbf{A}_1),\mathcal{K}_{\mathbf{A}_1}, R_{i_{1}},R_{d_{1}} \rangle\).
\end{theorem}
\begin{proof}
Observe that from Corollary 3 of \cite{CMZ}, $\langle \mathcal{F}_{R_h}, \mathcal{M}
\rangle$ is an $S$-space (isomorphic to $\langle \mathcal{X}(\mathbf{A}_2), \mathcal{K}_{\mathbf{A}_2}\rangle$). Additionally, from Theorem 16 of \cite{CMZ} it follows that, \( R_{i_1}[R_h(P)] \) and \( R_{d_1}[R_h(P)] \) are \( \mathcal{M} \)-increasing for all \( P \in \mathcal{X}(\mathbf{A}_2) \). Finally, from Theorem \ref{SLata Vietoris iso}, we have that $\langle \mathcal{F}_{R_h}, \mathcal{M},I_h, D_h \rangle$ is an SLata-space (isomorphic to \(\langle \mathcal{X}(\mathbf{A}_2),\mathcal{K}_{\mathbf{A}_2}, R_{i_{2}},R_{d_{2}} \rangle\)). Hence, by Definition \ref{SLata-Vietoris-type}, the result holds.   
\end{proof}
\\

The following result is an immediate consequence of Theorem \ref{hom sobre a SLata-Vietoris} and establishes the method of obtaining an SLata-Vietoris family from a congruence of SLata's.

\begin{corollary}\label{congruencia a SLata-Vietoris}
     Let $\langle\mathbf{A},i,d\rangle$ be an SLata, and let $\theta \subseteq A \times A$ be a congruence on it. Then, the following family:
     \[
\mathcal{F}_{\theta} = \{ R_{q_{\theta}}(Q) : Q \in \mathcal{X}(\mathbf{A} / \theta) \}, 
\]
where $q_{\theta} : A \rightarrow A / \theta$ denotes the canonical map, is non-empty subset of $ C_{\mathcal{K}_{\mathbf{A}}}(\mathcal{X}(\mathbf{A}))$ that constitutes a SLata-Vietoris family of \(\langle \mathcal{X}(\mathbf{A}),\mathcal{K}_{\mathbf{A}}, R_{i},R_{d} \rangle\).
\end{corollary}

Now we are able to state the main result of this section. It shows that there is a bijective correspondence between congruences of SLata's and SLata-Vietoris families of SLata-spaces. Its proof is analogue to the proofs of Theorems 10, 11 and 12, of \cite{CMZ}, respectively, so we omit them.

\begin{theorem}\label{Congruences ELatas}
Let $\langle \mathbf{A},i,d \rangle$  be an SLata. Then, the following hold:
\begin{enumerate}
        \item If $\theta$ is a congruence on $\langle \mathbf{A},i,d \rangle$, and $\mathcal{F}_\theta$ is the SLata-Vietoris family of its dual SLata-space obtained in Corollary \ref{congruencia a SLata-Vietoris}, then $\theta = \theta_{\mathcal{F}_\theta}$.
        \item If $\mathcal{F}$ is an SLata-Vietoris family of the dual $SLata$ space of $\mathbf{A}$, then 
\[
\theta_\mathcal{F} = \{(a, b) \in A^2 : [\beta(a)^{c}]_{\mathcal{F}}^{-} = [\beta(b)^{c}]_{\mathcal{F}}^{-}\}. 
\]
   is a congruence on $\langle \mathbf{A},i,d \rangle $ such that $\mathcal{F} = \mathcal{F}_{\theta_\mathcal{F}}$. Moreover, $\theta_{\mathcal{F}}$ coincides with the kernel of the homomorphism $\Box_{T_{\mathcal{F}}}\beta:\mathbf{A}\to S(\mathcal{F})$.
\end{enumerate}
Therefore, there is a one-to-one correspondence between the congruences of $\langle \mathbf{A},i,d\rangle$ and the SLata-Vietoris families of its dual SLata-space.
\end{theorem}

We conclude this section with a diagram that summarizes the workings of Theorem \ref{Congruences ELatas}. Specifically, the left column illustrates the entire process described in Theorem \ref{Congruences ELatas} (1), showing how, from a congruence of an SLata, we derive the associated SLata-Vietoris family of its dual SLata-space. Conversely, the right column presents the analogous process described in Theorem \ref{Congruences ELatas} (2), which reverses the steps in the left column. Here, $\mathcal{V}(\mathcal{X}(\mathbf{A}))$ represents the family of SLata-Vietoris families of the dual SLata-space of an SLata $\langle \mathbf{A},i,d\rangle$, and $\mathsf{Con}(\langle \mathbf{A},i,d\rangle)$ denote its set of congruences.

\begin{displaymath}
    \begin{array}{cccc}
        \begin{array}{c}
             \theta\in \mathsf{Con}(\langle \mathbf{A},i,d\rangle) \\
             \hline
             \\
             q_{\theta}:A\to A/\theta\\
             \hline
             \\
             R_{q_\theta}\subseteq \mathcal{X}(\mathbf{A}/\theta)\times \mathcal{X}(\mathbf{A}), \\
             \hline
             \\
             \mathcal{F}_{R_{q_\theta}}\in \mathcal{V}(\mathcal{X}(\mathbf{A}))
         \end{array} & & & \begin{array}{c}
             \mathcal{F}\in \mathcal{V}(\mathcal{X}(\mathbf{A}))\\
             \hline
             \\
             T_{\mathcal{F}}\subseteq \mathcal{F}\times \mathcal{X}(\mathbf{A})\\
             \hline
             \\            \Box_{T_{\mathcal{F}}}:S(\mathbf{A})\to S(\mathcal{F})\\
             \hline
             \\    
             ker(\Box_{T_{\mathcal{F}}}\beta)\in \mathsf{Con}(\langle \mathbf{A},i,d\rangle)
        \end{array} \\
    \end{array}
\end{displaymath}

\section{Ewald-Semilattices with adjunctions}\label{Ewald-Semilattices with adjunctions} 

Most algebraic structures related to tense logic are algebraic structures with a semilattice reduct, endowed with a quadruple of unary connectives that behave as pairs of adjunctions (or Galois connections). Examples include Tense algebras, IKt-algebras \cite{Figallo14}, Tense MV-algebras \cite{DG0207}, Lattices with adjunctions \cite{MS2014}, Tense distributive lattices with implication \cite{PZ2023}, and Negative Heyting algebras \cite{APZ2023}, to mention a few. Thus, it is natural to study meet-semilattices equipped with pairs of adjunctions that are suitably related. We refer to these structures as Ewald-Semilattices with adjunctions (ESLatas). These algebras are of particular interest because they provide a sufficiently general framework to encompass several structures within the context of tense logics.

\

The goal of this section is to demonstrate that the techniques developed so far for obtaining a categorical duality for SLatas can be successfully adapted to achieve a categorical duality for ESLatas.

\begin{definition}\label{Definicion ELata}
Let $\mathbf{A}=\langle A,\wedge,1\rangle$ be a semilattice. An Ewald-Semilattice with adjunctions (E-SLata, for short) is a structure $\langle\mathbf{A},P,G,F,H\rangle$ such that \( G \), \( H \), \( F \), and \( P \) are unary operators on \( A \) satisfying the following conditions:

\begin{itemize}
\item[{\rm (T1)}] \( P \dashv G \), 
\item[{\rm (T2)}] \( F \dashv H \), 
\item[{\rm (T3)}] \( G(x) \land F(y) \leq F(x \land y) \) and \( H(x) \land P(y) \leq P(x \land y) \).
\end{itemize}

We refer to \( G, H, F, \) and \( P \) as tense operators.
\end{definition}

It is readily seen from the previous Definition that the class of E-SLatas is a variety.

\begin{example}\label{fig1} Let us consider the semilattice  ${\bf A}$ visualized in Fig. \ref{fig: 1}.

\vspace{1.25cm}

\begin{figure}[htbp]
\begin{center}
\hspace{0.25cm}
\begin{picture}(-50,50)(0,0)

\put(00,00){\makebox(1,1){$\bullet$}}
\put(-20,20){\makebox(1,1){$\bullet$}}
\put(20,20){\makebox(1,1){$\bullet$}}
\put(00,40){\makebox(1,1){$\bullet$}}
\put(-40,40){\makebox(1,1){$\bullet$}}
\put(40,40){\makebox(1,1){$\bullet$}}
\put(-20,60){\makebox(1,1){$\bullet$}}
\put(20,60){\makebox(1,1){$\bullet$}}
\put(00,80){\makebox(1,1){$\bullet$}}
\put(00,00){\line(1,1){40}}
\put(00,00){\line(-1,1){40}}
\put(-40,40){\line(1,1){40}}
\put(40,40){\line(-1,1){40}}
\put(-20,20){\line(1,1){40}}
\put(20,20){\line(-1,1){40}}
\put(00,-10){\makebox(2,2){$ 0$}}
\put(00,90){\makebox(2,2){$ 1$}}
\put(-30,20){\makebox(2,2){$ a$}}
\put(30,20){\makebox(2,2){$ b$}}
\put(10,40){\makebox(2,2){$ d$}}
\put(-50,40){\makebox(2,2){$ c$}}
\put(50,40){\makebox(2,2){$ e$}}
\put(-30,60){\makebox(2,2){$ f$}}
\put(30,60){\makebox(2,2){$ g$}}
\end{picture}\caption{}
 \label{fig: 1}
\end{center}
\end{figure}

We define the operators $G$, $H$, $F$ and $P$ by the following table:

\begin{center}
\begin{tabular}{|c|c|c|c|c|c|c|c|c|c|}\hline
{\scriptsize  x}& {\scriptsize 0}& {\scriptsize  a} & {\scriptsize  b}  & {\scriptsize  c}  & {\scriptsize   d}  & {\scriptsize  e}  & {\scriptsize  f}  & {\scriptsize g}  & {\scriptsize  1}  \\ \hline
 {\scriptsize G(x)}  & {\scriptsize  0}   & {\scriptsize 0}  & {\scriptsize b}  & {\scriptsize 0} & {\scriptsize d}  & {\scriptsize  e} & {\scriptsize  d}  & {\scriptsize g} &   {\scriptsize 1} \\ 

{\scriptsize H(x)} & {\scriptsize  0}  & {\scriptsize a} & {\scriptsize 0} & {\scriptsize c} & {\scriptsize d} & {\scriptsize 0} & {\scriptsize f} & {\scriptsize  d } &  {\scriptsize 1 }                                                               \\ 

{\scriptsize F(x)} &  {\scriptsize 0} & {\scriptsize a} & {\scriptsize d} & {\scriptsize c} & {\scriptsize d} & {\scriptsize 1} & {\scriptsize f} & {\scriptsize 1} & {\scriptsize 1} \\ 

{\scriptsize P(x)} & {\scriptsize  0}  & {\scriptsize d}  & {\scriptsize b} & {\scriptsize 1} & {\scriptsize d} & {\scriptsize e} & {\scriptsize 1} & {\scriptsize g} & {\scriptsize 1} \\ \hline
\end{tabular}
\end{center}
\vspace{2mm}

It is easy to see that  $\langle{\bf A},G,H,F,P\rangle$ is an E-SLata.
\end{example}

\begin{lemma}
Let $\mathbf{A}$ be a semilattice and let $G, H, F, P$ be unary operators on $A$. The following conditions are equivalent.

\begin{enumerate}
    \item [(a)] $\langle \mathbf{A}, P,G,F,H \rangle$ is an E-SLata.
    
    \item [(b)]The following hold:

        \item[b1)] $Gx \wedge Fy \leq F(x \wedge y)$ and $Hx \wedge Py \leq P(x \wedge y)$,
        \item[b2)] $x \leq GPx$ and $x \leq HFx$,
        \item[b3)] $PGx \leq x$ and $FHx \leq x$.
    \end{enumerate}

\end{lemma}

\begin{proposition}
    Let $ \langle \mathbf{A},P,G,F,H \rangle$ an E-Lata. Then
    \begin{enumerate}
        \item [a)] $F(x \land y) \leq F(x) \land F(y) $ and  $ P(x \land y) \leq P(x) \land P(y) $
        \item[b)] $ x \land F(y) \leq F(P(x) \land y)$ and $ x \land P(y) \leq P(F(x) \land y) $
        \item[c)] $ F = FHF $, $ P = PGP $, $ G = GPG $ , $ H = HFH $

    \end{enumerate}
\end{proposition}
\begin{proof}
    Notice that (a) and (c) follows from (T1), (T2) and the fact that $F$ and $P$ are monotone. For (b), observe that by (T1), (T3) and Remark \ref{condicion adjuncion ecuacional}, we obtain  \[x\wedge F(y)\leq G(P(x))\wedge F(y)\leq F(P(x)\wedge y).\]
    The proof of the remaining inequality is similar. 
\end{proof}

\ 

We write $\mathsf{ESLata}$ for the category whose objects are E-SLatas, and whose morphisms are semilattice homomorphisms that preserve all the tense operators. Composition is defined as the usual composition of functions, and the identity morphism is the identity map. When no clarification is needed, we also write $\mathsf{ESLata}$ for the corresponding variety.

\

 In what follows, we will present the topological definitions that will allow us to establish the desired duality:
\begin{definition}
     A structure $\langle X,\mathcal{K},I_{1}, D_{1}, I_{2}, D_{2}\rangle$ is an E-SLata space if verified: 
     \begin{itemize}         
         \item [E1)] $\langle X,\mathcal{K},I_{1},D_{1}\rangle$ and $\langle X,\mathcal{K},I_{2},D_{2}\rangle$ are SLata-spaces.
         \item[E2)] For every $x \in X$, and  $U,V \in S(X)$, if $D_{1}(x)\subseteq L_{U}$ and $I_{2}(x)\subseteq L_{V}$ then $I_{2}(x)\subseteq L_{U \cap V}$.
         \item[E3)]  For every $x \in X$, and  $U,V \in S(X)$,  if $D_{2}(x)\subseteq L_{U}$ and $I_{1}(x)\subseteq L_{V}$ then $I_{1}(x)\subseteq L_{U \cap V}$.
     \end{itemize}
\end{definition}

Let $\langle X_1,\mathcal{K}_1, I_{1}, D_{1}, I_{2}, D_{2}\rangle$ and $\langle X_2,\mathcal{K}_2, J_{1}, M_{1}, J_{2}, M_{2}\rangle$ be two ESLata-spaces. We say that a relation $T \subseteq X_1 \times X_2$ is an \emph{ESLata-relation} if it is an SLata-relation between the SLata-spaces $\langle X_1, \mathcal{K}_1, I_{l}, D_{l} \rangle$ and $\langle X_2, \mathcal{K}_2, J_{l}, M_{l} \rangle$, for $l = 1, 2$. We write $\mathsf{ESLataSp}$ for the category whose objects are ESLata-spaces and whose morphisms are ESLata-relations. Composition is defined as the composition of meet-relations, as described in (\ref{Definicion composition meet-relations}), and the identity arrow is given by the dual of the specialization order. It is not hard to see that both the composition and the identity are well-defined, and the composition is associative.

\begin{proposition}\label{3}
    $\langle X,\mathcal{K}, I_{1}, D_{1}, I_{2}, D_{2}\rangle$ 
    is an E-SLata-space if and only if 
    \[
    \langle \mathbf{S}(X), m_{I_{1}}, m_{D_{1}}, m_{I_{2}}, m_{D_{2}} \rangle
    \]
    is an E-SLata.
\end{proposition}
\begin{proof}
    Let us assume that $\langle X,\mathcal{K},I_{1},D_{1},I_{2},D_{2}\rangle$ is an E-SLata-space. By (E1) and Proposition \ref{SLata-Sp to Slata}, we have that $m_{I_1}\dashv m_{D_1}$ and $m_{I_2}\dashv m_{D_2}$, so (T1) and (T2) holds. In order to prove (E3), let us take, $U, V \in S(X)$. We need to show that $m_{D_{1}}(U) \cap m_{I_{2}}(V) \subseteq m_{I_{2}}(U\cap V)$ and $m_{D_{2}}(U) \cap m_{I_{1}}(V) \subseteq m_{I_{1}}(U\cap V)$. To this end, let $x \in m_{I_{1}}(U) \cap m_{I_{2}}(V)$. Then $D_{1}(x)\subseteq L_{U}$ and $I_{2}(x)\subseteq L_{V}$ so, from (E2) it follows that $I_{2}(x)\subseteq L_{U\cap V}$,  thus $x \in m_{I_{2}}(U \cap V)$. The proof of the other inclusion is analogous. On the other hand, suppose that  $\langle \mathbf{S}(X), m_{I_{1}},m_{D_{1}},m_{I_{2}},m_{D_{2}} \rangle$ is an E-SLata, and let $U, V \in S(X)$. Then, from Proposition \ref{SLata to Slata-Sp}, both $\langle X, \mathcal{K},I_1,D_1\rangle$ and $\langle X, \mathcal{K},I_2,D_2\rangle$ are SLata-spaces, thus (E1) holds. Now, let $x\in X$ be such that $D_{1}(x)\subseteq L_{U}$ and $I_{2}(x)\subseteq L_{V}$. Then $x \in m_{D_{1}(U)}$ and $x \in m_{I_{2}}(V)$. Thus, $x \in m_{D_{1}(U)} \cap m_{I_{2}}(V)$. So, from the assumption we get that $x \in m_{I_{2}}(U \cap V)$. Then, $I_{2}(x)\subseteq L_{U \cap V}$, hence (E2) holds. The proof of (E3) is similar. This proves our claim.

\end{proof}

\begin{proposition}
    $ \langle \mathbf{A},P, G, F, H \rangle$ is an E-SLata if and only if $\langle \mathcal{X}(\mathbf{A}),R_P,R_G, R_F, R_H\rangle$ is an E-SLata space.
\end{proposition}

\begin{proof}
On the one hand, suppose that $ \langle \mathbf{A},P, G, F, H \rangle$ is an E-SLata. From Proposition \ref{SLata to Slata-Sp}, we know that both $\langle \mathcal{X}(\mathbf{A}),R_P,R_G\rangle$ and $\langle \mathcal{X}(\mathbf{A}), R_F, R_H\rangle$ are SLata-spaces, so (E1) holds. Now we prove (E2). For this, we need to show that for every $a, b \in A$ and $P \in \mathcal{X}(\mathbf{A})$, if $R_{G}(P) \subseteq L_{\beta(a)}$ and $R_{F}(P) \subseteq L_{\beta(b)}$, then $R_{F}(P) \subseteq L_{\beta(a \wedge b)}$. From (T3) and the fact that $\beta: A \to S(X(A))$ is an isomorphism, we have $\beta (G(a)) \cap \beta(F(b)) \subseteq \beta (F(a \wedge b))$ for every $a, b \in A$. Furthermore, by Lemma 10 of \cite{CMZ}, it follows that $m_{R_{G}}(\beta(a)) \cap m_{R_{F}}(\beta(b)) \subseteq m_{R_{F}}(\beta(a \wedge b))$. Now, let $P \in \mathcal{X}(\mathbf{A})$ and suppose $R_{G}(P) \subseteq L_{\beta(a)}$ and $R_{F}(P) \subseteq L_{\beta(b)}$. By the previous discussion, we immediately obtain $R_{F}(P) \subseteq L_{\beta(a \wedge b)}$. The proof of (E3) is analogous. On the other hand, suppose that $\langle \mathcal{X}(\mathbf{A}),R_P,R_G, R_F, R_H\rangle$ is an E-SLata space. From Proposition \ref{SLata to Slata-Sp}, we know that $P \dashv G$ and $F \dashv H$. As a direct consequence of (E2), we have that for all $a, b \in A$, $m_{R_{G}}(\beta(a)) \cap m_{R_{F}}(\beta(b)) \subseteq m_{R_{F}}(\beta(a \wedge b))$. Then, by Lemma 10 of \cite{CMZ}, it follows that $\beta (G(a) \wedge F(b)) \subseteq \beta (F(a \wedge b))$. Since $\beta$ is a poset isomorphism, we conclude that $G(a) \wedge F(b) \leq F(a \wedge b)$. Similarly, we can prove that $H(a) \wedge P(b) \leq P(a \wedge b)$. Thus, (T3) holds, and we conclude that $ \langle \mathbf{A},G, H, F, P \rangle$ is an E-SLata, as desired.

\end{proof}
\ 

In a similar fashion to what we did in Section \ref{Semilattices with an adjunction}, we can adapt the arguments given immediately after Proposition \ref{unit duality} to construct a proper functor from $\mathsf{ESLata}^{op}$ to $\mathsf{ESLataSp}$, defined as follows:
\[
\begin{array}{rcl}
   \langle\mathbf{A}, P, G, F, H\rangle & \mapsto & \langle \mathcal{X}(\mathbf{A}), \mathcal{K}_\mathbf{A}, R_P, R_G, R_F, R_H \rangle, \\
   h & \mapsto & R_h.
\end{array}
\]

Furthermore, this adaptation can also be extended to define a functor from $\mathsf{ESLataSp}$ to $\mathsf{ESLata}^{op}$ as follows:
\[
\begin{array}{rcl}
    \langle X,\mathcal{K},I_1,D_1,I_2,D_2\rangle & \mapsto & \langle \mathbf{S}(X),m_{I_1},m_{D_1},m_{I_2},m_{D_2}\rangle, \\
     T & \mapsto & \Box_{T}.
\end{array}
\]

The following results are the ESLata versions of Propositions \ref{unit duality} and \ref{counit duality}. Their proofs are analogous to those presented in Section \ref{Semilattices with an adjunction}, so they are omitted.

\begin{proposition}\label{unit dualityESLata}
  Let  $\langle\mathbf{A},i,d\rangle$ be an ESLata, then $\beta: \mathbf{A} \rightarrow \mathbf{S}(\mathcal{X}(\mathbf{A}))$ is an isomorphism of ESLatas.
\end{proposition}

\begin{proposition}\label{counit dualityESLata}
    Let $\langle X,K,I_1,D_1,I_2,D_2\rangle $ be an ESLata-space. Then $T_{X}\subseteq \mathcal{X}(\mathbf{S}(X)) \times X$ defined by 
    \begin{displaymath}
        \begin{array}{ccc}
            (H_{X}(y), x) \in T_X & \Leftrightarrow  & H_{X}(y)\subseteq H_{X}(x)
        \end{array}
    \end{displaymath}
 is an isomorphism of ESLata-spaces. 
\end{proposition}

Therefore, as consequence of Propositions \ref{unit dualityESLata} and \ref{counit dualityESLata} we obtain:
\begin{theorem}\label{dualidad ESLata}
    The categories $\mathsf{ELata}$ and $\mathsf{EPLata}$ are dually equivalent.
\end{theorem}

-------------------------------------------------------------------------------------------------------
\\
Belén Gimenez, \\
Departamento de Matem\'atica,\\
Facultad de Ciencias Exactas (UNLP),\\
50 y 115, La Plata (1900)\\
belengim.28@gmail.com

-----------------------------------------------------------------------------------
\\
Gustavo Pelaitay, \\
Facultad de Filosofía Humanidades y Artes,\\
Instituto de Ciencias Básicas (UNSJ),\\
Av.José Ignacio de la Roza, San Juan (5400),\\
and CONICET, Argentina.\\
gpelaitay@gmail.com

-----------------------------------------------------------------------------------------
\\
William Zuluaga,\\
Facultad de Ciencias Exactas (UNCPBA),\\
Pinto 399, Tandil (7000),\\
and CONICET, Argentina,\\
wizubo@gmail.com

\end{document}